\documentclass[10pt,a4paper]{article}
\usepackage{amssymb,amsmath}
\usepackage{amsthm}
\usepackage{bbm}
\usepackage{stmaryrd}
\usepackage{hyperref}

\input xy
\xyoption{all} 

\numberwithin{equation}{section}

\newtheorem{theorem}{Theorem}[section]

\newtheorem{lemma}[theorem]{Lemma}
\newtheorem{proposition}[theorem]{Proposition}

\theoremstyle{definition}
\newtheorem{definition}[theorem]{Definition}
\newtheorem{remark}[theorem]{Remark}
\newtheorem{example}[theorem]{Example}

\newcommand{\rmd}{\textnormal{d}}

\DeclareMathOperator{\Vect}{Vect}

\DeclareMathOperator{\ad}{ad}

\font\black=cmbx10 \font\sblack=cmbx7 \font\ssblack=cmbx5 \font\blackital=cmmib10  \skewchar\blackital='177
\font\sblackital=cmmib7 \skewchar\sblackital='177 \font\ssblackital=cmmib5 \skewchar\ssblackital='177
\font\sanss=cmss10 \font\ssanss=cmss8 
\font\sssanss=cmss8 scaled 600 \font\blackboard=msbm10 \font\sblackboard=msbm7 \font\ssblackboard=msbm5
\font\caligr=eusm10 \font\scaligr=eusm7 \font\sscaligr=eusm5  \font\fraktur=eufm10
\font\sfraktur=eufm7 \font\ssfraktur=eufm5 
\font\bsymb=cmsy10 scaled\magstep2
\def\all#1{\setbox0=\hbox{\lower1.5pt\hbox{\bsymb
       \char"38}}\setbox1=\hbox{$_{#1}$} \box0\lower2pt\box1\;}
\def\exi#1{\setbox0=\hbox{\lower1.5pt\hbox{\bsymb \char"39}}
       \setbox1=\hbox{$_{#1}$} \box0\lower2pt\box1\;}

\newfam\bifam
\textfont\bifam=\blackital \scriptfont\bifam=\sblackital \scriptscriptfont\bifam=\ssblackital

\newfam\blfam
\textfont\blfam=\black \scriptfont\blfam=\sblack \scriptscriptfont\blfam=\ssblack

\newfam\bbfam
\textfont\bbfam=\blackboard \scriptfont\bbfam=\sblackboard \scriptscriptfont\bbfam=\ssblackboard

\newfam\ssfam
\textfont\ssfam=\sanss \scriptfont\ssfam=\ssanss \scriptscriptfont\ssfam=\sssanss
\def\sss#1{{\fam\ssfam\relax#1}}

\newfam\clfam
\textfont\clfam=\caligr \scriptfont\clfam=\scaligr \scriptscriptfont\clfam=\sscaligr

\newfam\frfam
\textfont\frfam=\fraktur \scriptfont\frfam=\sfraktur \scriptscriptfont\frfam=\ssfraktur

\def\hpb#1{\setbox0=\hbox{${#1}$}
    \copy0 \kern-\wd0 \kern.2pt \box0}
\def\vpb#1{\setbox0=\hbox{${#1}$}
    \copy0 \kern-\wd0 \raise.08pt \box0}

\def\pmb#1{\setbox0\hbox{${#1}$} \copy0 \kern-\wd0 \kern.2pt \box0}
\def\pmbb#1{\setbox0\hbox{${#1}$} \copy0 \kern-\wd0
      \kern.2pt \copy0 \kern-\wd0 \kern.2pt \box0}
\def\pmbbb#1{\setbox0\hbox{${#1}$} \copy0 \kern-\wd0
      \kern.2pt \copy0 \kern-\wd0 \kern.2pt
    \copy0 \kern-\wd0 \kern.2pt \box0}
\def\pmxb#1{\setbox0\hbox{${#1}$} \copy0 \kern-\wd0
      \kern.2pt \copy0 \kern-\wd0 \kern.2pt
      \copy0 \kern-\wd0 \kern.2pt \copy0 \kern-\wd0 \kern.2pt \box0}
\def\pmxbb#1{\setbox0\hbox{${#1}$} \copy0 \kern-\wd0 \kern.2pt
      \copy0 \kern-\wd0 \kern.2pt
      \copy0 \kern-\wd0 \kern.2pt \copy0 \kern-\wd0 \kern.2pt
      \copy0 \kern-\wd0 \kern.2pt \box0}

\mathchardef\za="710B  
\mathchardef\zb="710C  
\mathchardef\zg="710D  
\mathchardef\zd="710E  
\mathchardef\zve="710F 
\mathchardef\zz="7110  
\mathchardef\zh="7111  
\mathchardef\zvy="7112 
\mathchardef\zi="7113  
\mathchardef\zk="7114  
\mathchardef\zl="7115  
\mathchardef\zm="7116  
\mathchardef\zn="7117  
\mathchardef\zx="7118  
\mathchardef\zp="7119  
\mathchardef\zr="711A  
\mathchardef\zs="711B  
\mathchardef\zt="711C  
\mathchardef\zu="711D  
\mathchardef\zvf="711E 
\mathchardef\zq="711F  
\mathchardef\zc="7120  
\mathchardef\zw="7121  
\mathchardef\ze="7122  
\mathchardef\zy="7123  
\mathchardef\zf="7124  
\mathchardef\zvr="7125 
\mathchardef\zvs="7126 
\mathchardef\zf="7127  
\mathchardef\zG="7000  
\mathchardef\zD="7001  
\mathchardef\zY="7002  
\mathchardef\zL="7003  
\mathchardef\zX="7004  
\mathchardef\zP="7005  
\mathchardef\zS="7006  
\mathchardef\zU="7007  
\mathchardef\zF="7008  
\mathchardef\zW="700A  
\mathchardef\zC="7009  

\newcommand{\be}{\begin{equation}}
\newcommand{\ee}{\end{equation}}

\newcommand{\bea}{\begin{eqnarray}}
\newcommand{\eea}{\end{eqnarray}}
\def\*{{\textstyle *}}

\newcommand{\s}{{\textstyle *}}

\def\Sec{\sss{Sec}}

\def\sT{{\sss T}}

\def\st{{\sss t}}

\def\s*{{\scriptstyle *}}

\newcommand{\beas}{\begin{eqnarray*}}
\newcommand{\eeas}{\end{eqnarray*}}

\def\half{\frac{1}{2}}

\begin{document}
\bibliographystyle{plain}

\author{
        Andrew James Bruce\\
        {\it Institute of Mathematics}\\
                {\it Polish Academy of Sciences}
                }

\date{\today}
\title{Killing sections and sigma models\\ with Lie algebroid targets\thanks{Research partially funded  by the  Polish National Science Centre grant under the contract number DEC-2012/06/A/ST1/00256.  }}
\maketitle

\begin{abstract}
\noindent We define and examine the notion of a Killing section of a Riemannian Lie algebroid as a natural generalisation of a Killing vector field. We show that the various expression for a vector field to be Killing naturally generalise to the setting of Lie algebroids. As an application we examine the internal symmetries of  a class of sigma models for which the target space is a Riemannian Lie algebroid. Critical points of these sigma models are interpreted as generalised harmonic maps.
\end{abstract}

\begin{small}
 \smallskip

\end{small}

\section{Introduction}\label{sec:intro}
Lie groupoids and Lie algebroids are fundamental concepts in differential geometry.  Lie groupoids provide a unifying framework to discuss diverse topics in modern differential geometry including the theory of group actions, foliations, Poisson geometry, orbifolds, principal bundles, connection theory and so on.  The infinitesimal counterpart to Lie groupoids are Lie algebroids; one should keep in mind the relation between Lie groups and Lie algebras. For the general theory of Lie groupoids and Lie algebroids the reader should consult \cite{Mackenzie:2005}.

Recall the standard notion of a Lie algebroid as a vector bundle $E \rightarrow M$ equipped with a Lie bracket on the sections $[\bullet, \bullet]: \Sec(E) \times \Sec(E) \rightarrow \Sec(E)$ together with an anchor $\rho: \Sec(E) \rightarrow \Vect(M)$ that satisfy the Leibniz rule
\begin{equation*}
\nonumber [u,fv] = \rho(u)[f] \: v  +   f [u,v],
\end{equation*}
\noindent for all $u,v \in \Sec(E)$ and $f \in C^{\infty}(M)$. The Leibniz rule implies that the anchor is actually a Lie algebra morphism: $\rho\left([u,v]\right) = [\rho(u), \rho(v)]$. If we pick some local basis for the sections $(s_{a})$, then the structure functions of a Lie algebroid are defined by
\begin{align}
&[s_{a} , s_{b}] = Q_{ab}^{c}(x)s_{c},&
& \rho(s_{a})= Q_{a}^{A}(x)\frac{\partial}{\partial x^{A}},&
\end{align}
\noindent and satisfy
\begin{align*}
&Q_{ab}^{c} + Q_{ba}^{c} =0,&
& Q_{a}^{A}\frac{\partial Q_{b}^{B}}{\partial x^{A}} - Q_{b}^{A}\frac{\partial Q_{a}^{B}}{\partial x^{A}}=0,&
& \sum_{\textnormal{cyclic}}\left( Q_{a}^{A}\frac{\partial Q_{bc}^{d}}{\partial x^{A}} - Q_{ae}^{d}Q_{bc}^{e} \right)=0.&
\end{align*}

The mantra of Lie algebroids is: \emph{what ever you can do on the tangent bundle of a manifold you can do on a Lie algebroid}.  Indeed, the tangent bundle of a manifold is one of the two fundamental examples of a Lie algebroid, the second example is a Lie algebra. The theory of Lagrangian mechanics on Lie algebroids is now well-developed. The Riemannian geometry of Lie algebroids is less explored than the applications of Lie algebroids to mechanics. The notion of a Riemannian structure on a Lie algebroid is just that of a metric on the underlying vector bundle; there is no compatibility condition required. However, the Lie algebroid structure allows for a more interesting theory that parallels the classical theory of Riemannian manifolds closely.  Interestingly, all the main constructions from Riemannian geometry pass to Lie algebroids, in particular a clear notion of torsion as well as  Levi--Civita connections.

We direct the interested reader to the following papers \cite{Anastasiei:2010,Boucetta:2011,Cortes:2004,Grabowska:2006,Jozwikowski:2012}, all of which discuss Riemannian geometry on Lie algebroids to varying extents. The notion of a Riemannian metric on a Lie groupoid is rather delicate  and various notions exists, one has to think about the compatibility of the groupoid and metric structures. We will not touch upon Riemannian Lie groupoids and direct the interested reader to the recent publication \cite{Hoyo:2015} and references therein.

One notion that appears to be missing, or at least not properly discussed, is that of a \emph{Killing section} of a Lie algebroid equipped with a Riemannian metric. Such a section would be the natural generalisation of a Killing vector field in standard Riemannian geometry. In this paper we present the notion of a Killing section, examine the various ways to express this and some direct consequences.

We remark that Killing vectors and Killing tensors have been fundamental in the development of general relativity. As far as we  are aware, Anastasiei and G\^{\i}r\c{t}u \cite{Anastasiei:2014} were the first to introduce the Einstein field equations (without cosmological constant) on a Riemannian Lie algebroid. However, it is unclear if there are any deep physical applications of this direct generalisation. The notion of a Lie algebroid has of course  been applied to the theory of gravity, just not in such a direct way;   see for example \cite{Strobl:2004,Vacaru:2012}.

In this paper we also define a class of sigma models for which the target manifold is a Riemannian Lie algebroid and show how Killing sections describe their internal symmetries. As far as we know, these models have not been studied before. The equations of motion for these models follows from the work of  Mart\'{\i}nez \cite{Martinez:2005} and the critical points can naturally be interpreted as generalised harmonic maps.  We comment that (non-linear) sigma models generally represent a very rich class of field theories that have found applications in high energy physics, condensed matter physics and string theory. From a mathematical perspective sigma models provide a deep link between quantum field theory and differential geometry. It is  natural and hopefully useful to study Lie algebroid versions of sigma models.

\smallskip

\noindent\textbf{Arrangement:} In section \ref{sec:metrics and coflows} we review the idea of a Riemannian metric and (co)geodesics in the setting of Lie algebroids. There are no new results on this section.  In section \ref{sec:killing} we proceed to the notion of a Killing section of a Lie algebroid equipped with a Riemannian metric  and explore the basic properties of such sections.  In section \ref{sec:SigmaModel} we briefly  apply the constructions in this note to a natural generalisation of standard sigma models in which the Riemannian structure on the target manifold gets replaced with a Riemannian structure on a Lie algebroid. We end this paper with a few closing comments in section \ref{sec:conclusion}.
\section{Riemannian metrics and cogeodesic flows on Lie algebroids}\label{sec:metrics and coflows}
In this section we will briefly recall some basic notions related to Riemannian structures on Lie algebroids. Nothing in this section is new.  For details the reader is advised to consult the literature listed in the introduction.\smallskip

\noindent \textbf{Metrics and connections:} Recall that a \emph{Riemannian metric} on a Lie algebroid $E \rightarrow M$ is a smooth assignment of an inner product to each of the fibres:
\begin{equation*}
\langle \bullet | \bullet \rangle : \Sec(E) \times \Sec(E) \rightarrow C^{\infty}(M).
\end{equation*}

In a local trivialisation we have $u = u^{a}(x)s_{a}$ and $v = v^{a}(x)s_{a} \in \Sec(E)$ and the Riemannian metric is given by
\begin{equation}
\langle u | v\rangle = G_{ab}(x)u^{a} v^{b}.
\end{equation}

A Lie algebroid equipped with a Riemannian metric shall be referred to as a \emph{Riemannian Lie algebroid} and simply denoted by $(E,G)$ when the Lie algebroid structure is clear. Any Riemannian metric on a Lie algebroid defines a quadratic function on $E$, which in local coordinates $(x^{A}, y^{a})$ is $G = y^{a}y^{b}G_{ba}$. Following common practice we will also refer to $G \in C^{\infty}(M)$ as the Riemannian metric.

One should note that the definition of a Riemannian metric on a Lie algebroid is just the definition of a Riemannian metric on the underlying vector bundle structure. Thus, as all vector bundles admit Riemannian metrics \emph{all} Lie algebroids admit Riemannian metrics.  The two extreme examples here are Riemannian metrics on a manifold $M$, which are really metrics on $\sT M$, and  non-degenerate scalar products on Lie algebras; for example the Killing from on a semisimple Lie algebra.  The general situation should be though of as a mixture of these two extremes.

 \begin{remark}
 The notion of a pseudo-Riemannian metric as found in general relativity also has the obvious generalisation to Lie algebroids. All the constructions in this note do not depend critically on the signature of the metric.
\end{remark}
Let us briefly recall the notion of a \emph{Lie algebroid connection} (we will refer to this simply as a connection) as a map
\begin{equation}
\nabla : \Sec(E)\times \Sec(E) \rightarrow \Sec(E),
\end{equation}
\noindent that satisfies the following conditions
\begin{align*}
& \nabla_{fu}v = f \nabla_{u}v,&
&\nabla_{u}(fv) = f \nabla_{u}v + \rho(u)[f]\: v,&
\end{align*}
\noindent for all $u, v \in \Sec(E)$ and $f \in C^{\infty}(M)$. In a chosen local bases the \emph{Christoffel symbols} are defined by $\nabla_{s_{a}}s_{b} = \Gamma_{ab}^{c}s_{c}$.  The \emph{curvature} and \emph{torsion} of a connection as defined as
\begin{align}\label{eqn:curvature and torsion}
&R(u,v)w  := \left [\nabla_{u}, \nabla_{v} \right]w - \nabla_{[u,v]}w,& &\textnormal{and}& &T(u,v) := \nabla_{u}v - \nabla_{v}u - [u,v].&
\end{align}
\noindent Naturally a connection with vanishing curvature is referred to as a \emph{flat connection} and a connection with vanishing torsion is said to be a \emph{torsionless connection}. The vanishing of the torsion of a connection  can be viewed as a compatibility condition between the connection and the Lie algebroid structure viz $[u,v] = \nabla_{u}v - \nabla_{v}u$.

Let us fix some metric on the Lie algebroid, then a  connection is said to be \emph{metric compatible} if the following holds
\begin{equation}\label{eqn:metric compat}
\rho(u)\langle v|w \rangle = \langle\nabla_{u} v|w \rangle + \langle v| \nabla_{u}w \rangle.
\end{equation}
Following the classical results we have:
\begin{theorem}\textnormal{(The Fundamental Theorem of Riemannian Lie Algebroids)}
There is a unique connection on a Riemannian Lie algebroid $(E, G)$ characterised by the two properties that it has vanishing torsion and be metric compatible. Such  connections are known as Levi--Civita connections.
\end{theorem}

The metric compatibility (\ref{eqn:metric compat}) and the vanishing of the torsion (\ref{eqn:curvature and torsion}) imply that the Levi--Civita connection for a given Riemannian metric is uniquely define by the \emph{Koszul formula} (see for example \cite{Anastasiei:2014})  
\begin{eqnarray*}
2 \langle \nabla_{u}v | w \rangle &=&  \rho(u)\langle v|w\rangle + \rho(v)\langle w|u\rangle - \rho(w)\langle u|v\rangle\\
\nonumber & -& \langle u|[v,w]\rangle  + \langle v|[w,u]\rangle + \langle w|[u,v]\rangle,
\end{eqnarray*}
\noindent for all $u, v$ and $w \in \Sec(E)$. 

In local coordinates the  Christoffel symbols for a Levi--Civita connection are given  by
\begin{equation}
\Gamma_{bc}^{a} := \frac{1}{2} \: G^{ad}\left(Q_{c}^{A}\frac{\partial G_{bd}}{\partial x^{A}} + Q_{b}^{A}\frac{\partial G_{cd}}{\partial x^{A}}  - Q_{d}^{A}\frac{\partial G_{bc}}{\partial x^{A}} + Q_{db}^{e}G_{ec} + Q_{dc}^{e}G_{eb}  {-} Q_{bc}^{e}G_{ed}\right).
\end{equation}

The local expression for the Christoffel symbols for the Levi--Civita connection on a general Lie algebroid should remind one of the local expression for the standard case of the tangent bundle written in a non-coordinate bases. Indeed, the proof of the fundamental theorem  follows via the standard arguments and is well-known.
\medskip

We state that  will  almost exclusively  be dealing with Levi--Civita connections in this note.

\smallskip

\noindent \textbf{Hamiltonians and cogeodesic flows:} As we are dealing with Lie algebroids the dual bundle $E^{*}$ comes with the structure of linear Poisson structure given in local coordinates $(x^{A}, \pi_{a})$ as
\begin{equation}
\{F,H\}_{E} = Q_{a}^{A}\left( \frac{\partial F}{\partial \pi_{a}} \frac{\partial H}{\partial x^{A}} - \frac{\partial F}{\partial x^{A}} \frac{\partial H}{\partial \pi_{a}} \right) - Q_{ba}^{c}\pi_{c} \frac{\partial F}{\partial \pi_{a}}\frac{\partial H}{\partial \pi_{b}}.
\end{equation}
\noindent for any $F$ and $H \in C^{\infty}(E^{*})$. In the above the pair $(Q_{a}^{A}(x) , Q_{ab}^{c}(x))$ are the structure functions of the Lie algebroid $E\rightarrow M$.

We will consider the Hamiltonian system on $E^{*}$, with the following Hamiltonian
\begin{equation}
\mathcal{H}(x, \pi) =\frac{1}{2}G^{ab}(x)\pi_{b}\pi_{a} \in C^{\infty}(E^{*}),
\end{equation}
\noindent where $G^{ac}G_{cb} = G_{bc}G^{ca} = \delta_{b}^{a}$ defines the `inverse structure' which we informally refer to as the \emph{energy}.

\begin{definition}
Let us fix some Riemannian Lie algebroid $(E, G)$. The Hamiltonian  flow on $E^{*}$ generated by $\mathcal{H}$ will be referred to as the \emph{cogeodesic  flow} of the Riemannian Lie algebroid in question. A curve $c: [0,1] \rightarrow E^{*}$ is called a \emph{cogeodesic} if and only if $X_{\mathcal{H}} \circ c = \st c$, where $X_{\mathcal{H}}$ is the Hamiltonian vector field associated with the energy and $\st c$ is the tangent lift of the curve $c$.
\end{definition}
Using the local expressions the cogeodesic flow is described by
\begin{align}\label{eqn:cogeodesic equations}
&\dot{x}^{A} = G^{ab}Q_{a}^{A}\pi_{b},& &\dot{\pi}_{a} = \left(  G^{bd}Q_{da}^{c} - \frac{1}{2}Q_{a}^{A}\frac{\partial G^{bc}}{\partial x^{A}}\right)\pi_{c}\pi_{b},&
\end{align}

\noindent where the `dot' has the meaning of time derivative (i.e. derivative with respect to the affine parameter used to describe the cogeodesics).

By using the metric to define the isomorphisms $E \simeq E^{*}$, i.e. $\pi_{b} = y^{a}G_{ab}$, where $(x^{A}, y^{a})$ are coordinates on $E$ the cogeodesic flow can be shown to be equivalent to the \emph{generalised geodesic equations}
\begin{align}\label{eqn:geodesic equations}
&\dot{x}^{A} =  y^{a}Q_{a}^{A},& &\dot{y}^{a} + \Gamma_{bc}^{a}y^{c}y^{b}=0.&
\end{align}

Curves $\gamma: [0,1] \rightarrow E$ that satisfy the generalised geodesic equations are referred to as \emph{geodesics}. From (\ref{eqn:geodesic equations}), we see that the curve $\gamma$ is \emph{admissible}; that is the virtual velocity $y$ is related to the actual velocity $\dot{x}$ via the anchor. Moreover the existence and uniqueness of geodesics given a starting point $(x_{0}, y_{0})$ is clear.  If the initial point is such that $y^{a}_{0}Q_{a}^{A}(x_{0})=0$, then the geodesic lies in $E_{x_{0}}$ and is said to be a  \emph{vertical geodesic}.      Because of the non-degeneracy of the Riemannian metric cogeodesics and geodesics are equivalent, but we prefer to make the distinction.

\section{Killing sections and symmetries}\label{sec:killing}
We now turn our attention to the concept of a Killing section. As all the statements in this section follow via direct computation in local coordinates rather than any clever arguments we omit detailed proofs. Indeed many of the proofs follows the standard proofs found in any textbook on Riemannian geometry.

\smallskip
\noindent \textbf{Killing sections and the Killing equation:}
 First we  recall the notion of the tangent lift of a section \cite{Grabowska:2006,Grabowski:1999}. Given a section $u = u^{a}s_{a}$ one can lift it to a vector field on $E$ as
\begin{equation}\label{eqn;lift}
\rmd_{\sT}(u) := u^{a}Q_{a}^{A}\frac{\partial}{\partial x^{A}} + \left(y^{a}Q_{a}^{A}\frac{\partial u^{c}}{\partial x^{A}} {-} y^{a}u^{b}Q_{ba}^{c}  \right)\frac{\partial}{\partial y^{c}} \in \Vect(E).
\end{equation}
One should note that the (infinitesimal) action of $\rmd_{\sT}(u)$ preserves the Lie algebroid structure.  We can then use this lift to define the Lie derivative of a Riemannian metric, understood as a quadratic function on $E$, along a section. In  particular we are led naturally to the following definition:
\begin{definition}\label{def:killing1a}
A section $u \in \Sec(E)$ of a Riemannian Lie algebroid $(E,G)$ is a \emph{Killing section} if and only if
\begin{equation}\label{eqn:killing1a}
\mathcal{L}_{u}G := \rmd_{\sT}(u)[G] =0.
\end{equation}
\end{definition}

\begin{lemma}\label{lem:local killing}
In local coordinates (\ref{eqn:killing1a}) given the following condition
\begin{equation}\label{eqn:killing2}
u^{a}Q_{a}^{A} \left( \frac{\partial G_{bc}}{\partial x^{A}}\right) + Q_{b}^{A}\left(\frac{\partial u^{d}}{\partial x^{A}}  \right)G_{dc} + Q_{c}^{A} \left( \frac{\partial u^{d}}{\partial x^{A}}\right) G_{db} - u^{a}Q_{ab}^{d}G_{dc} - u^{a}Q_{ac}^{d}G_{db}=0.
\end{equation}
\end{lemma}
For the case of $E = \sT M$ we have that $Q_{a}^{b} = \delta_{a}^{b}$ and $Q_{ab}^{c}=0$ and so (\ref{eqn:killing2})  clearly reduces to  the standard  Lie derivative along a vector field of the  Riemannian metric vanishing.

Following standard ideas about Lie algebroids, we identify sections of $E$ with linear functions on $E^{*}$, or in the graded language homogeneous functions of weight one. That is, in any local trivialisation,  we have the association
\begin{equation}
u^{a}(x)s_{a} \leftrightsquigarrow u^{a}(x)\pi_{a}.
\end{equation}

\noindent \textbf{Warning} In the following we will mean by $u$ a section of $E$ understood as a section or a linear function. The context should make the distinction clear. Moreover we will use the identification of the  Poisson bracket on $E^{*}$ restricted to weight one functions with the Lie bracket on sections of $E$ throughout the rest of this note.
\begin{theorem}\label{thm:killing1}
 Let us fix some Riemannian Lie algebroid $(E, G)$. Then a section $u \in \Sec(E)$ is a Killing section if and only if
 \begin{equation}\label{eqn:killing1b}
 \{ u, \mathcal{H}\}_{E} =0.
 \end{equation}
\end{theorem}
\begin{proof}
The theorem follows by explicit calculation in local coordinates using Lemma \ref{lem:local killing}.

\end{proof}

\smallskip
\noindent \textbf{Statement}  Killing sections represent `generalised symmetries' of the Riemannian metric $G$, or equivalently they correspond to `conserved charges' along the cogeodesic flow generated by the energy $\mathcal{H}$. This is of course in complete agreement with the classical case.
\smallskip

\begin{theorem}\label{thm:killing2}
A section $u \in \Sec(E)$ of a Riemannian Lie algebroid $(E,G)$ is a  Killing section if and only if
\begin{equation}\label{eqn:killing3}
\langle \nabla_{v}u | w \rangle + \langle \nabla_{w}u |v \rangle =0,
\end{equation}
\noindent for all $v,w \in \Sec(E)$.
\end{theorem}
\begin{proof}
The theorem follows by explicit calculation in local coordinates using Lemma \ref{lem:local killing}.

\end{proof}

For the case of $E = \sT M$ (\ref{eqn:killing3}) reduces to the standard Killing equation, $\nabla_{\mu}X_{\nu} +\nabla_{\nu}X_{\mu} =0 $, which is often taken as the \emph{definition} of a Killing vector. We will refer to (\ref{eqn:killing3}) the \emph{Lie algebroid Killing equation}.

As standard  the components of the curvature tensor  are given by  $R(s_{a}, s_{b})s_{c} := R_{a\:\: bc}^{\:\:d}s_{d}$ in some local basis of sections. It is not hard to show  by application of the Lie algebroid Killing equation and the symmetries of the curvature tensor that the following proposition holds:
\begin{proposition}\label{prop:finite vector space}
If a  section $u \in \Sec(E)$ of a Riemannian Lie algebroid $(E,G)$ is a Killing section then
\begin{equation*}
(\nabla_{b} \nabla_{c}u)^{d}G_{da} = - u^{e}R_{e\:\: ca}^{\:\:d}G_{db}.
\end{equation*}
\end{proposition}

A direct consequence of the previous proposition is the following:
\begin{proposition}\label{prop:diff equations}
A Killing section is completely determined by its value $u^{a}$ and  the value of the anti-symmetric tensor $L_{ab} := (\nabla_{a}u)^{c}G_{cb}$ at any point $p \in M$. That is, given these values at $p \in M$, the values at $q \in M$ are determined by integration of the following differential equations
\begin{align*}
&v^{b}(\nabla_{a}u)^{c}G_{cb} = v^{b}L_{ba},\\
& v^{b} (\nabla_{b} \nabla_{c}u)^{d}G_{da} = - v^{b}u^{e}R_{e\:\: ca}^{\:\:d}G_{db},
\end{align*}
\noindent for any admissible curve with base curve connecting $p$ and $q$. Here $\dot{x}^{A}(t) = v^{a}(t)Q_{a}^{A}(x(t))$ where $x^{A}(t)$ is the local coordinate representation of the base curve.
\end{proposition}
We then see that the maximum number of linearly independent  Killing sections is then determined by the dimension of the space of initial data $(u^{a}, L_{bc})$. Thus if the rank of the vector bundle $E$ is $n$, then the dimension of the space of initial data is $n(n+1)\slash 2$.

\smallskip

\noindent \textbf{Statement} From the definition of a Killing section it is clear that the set of Killing sections forms a vector space over the reals. Moreover, we know that this vector space is of finite dimension and is bounded by $n(n+1)\slash 2$. Following classical nomenclature, a Riemannian Lie algebroid is called \emph{maximally symmetric}  if it possesses the maximal number of linearly independent Killing sections.

\smallskip

\noindent \textbf{Geodesic sections:} Following classical notion of a geodesic vector field we arrive at the following definition:
\begin{definition}
A \emph{geodesic section} as a  section $v \in \Sec(E)$ that satisfies $\nabla_{v}v=0$.
\end{definition}
The nomenclature is apt as the generalised geodesic equations can be derived from $\nabla_{\gamma(t)}\gamma(t) =0$, upon insisting  that $\gamma$ being admissible.
\begin{proposition}
Provided $u$ is a Killing section and $v$ a geodesic section we have that
\begin{equation}
\rho(v)\langle u|v\rangle =  \langle\nabla_{v}u | v \rangle + \langle u | \nabla_{v} v \rangle =0.
\end{equation}
\end{proposition}
\begin{proof}
Observe that (\ref{eqn:killing3}) implies that  $\langle \nabla_{v}u | v \rangle =0$ if $u \in \Sec(E)$ is a Killing section. Using this observation and the metric compatibility establishes the proposition.
\end{proof}
\smallskip

\noindent \textbf{Statement} The above equation can then be interpreted as saying that $\langle u|v\rangle$, which is the component  of $v$ in the direction of $u$, is constant along the integral curves of $\rho(v) \in \Vect(M)$. Similarly, one can directly show that if $u$ is both Killing and geodesic then $\langle u |u\rangle$ must be constant.

\newpage

\noindent \textbf{The Lie algebra of Killing sections:} We need the following lemma that was first proved in \cite{Grabowski:1999}:
\begin{lemma}\label{lem:lie algebra}
The tangent lift of a section of a Lie algebroid is a Lie algebra morphism between the Lie algebra of sections of $E$ and vector fields on $E$:
\begin{equation*}
\rmd_{\sT}([u,v]) = [\rmd_{\sT}(u), \rmd_{\sT}(v)].
\end{equation*}
\end{lemma}
From Lemma \ref{lem:lie algebra} the following proposition is easily seen to be true:
\begin{proposition}
The set of Killing sections for a given Riemannian Lie algebroid $(E,G)$ forms a  Lie algebra over $\mathbb{R}$ of maximum dimension $n(n+1)\slash 2$, where $n$ is the rank of the vector bundle $E$.
\end{proposition}
Note that, exactly in accordance to the standard case, we do not have a Lie subalgebroid as the Killing sections do not form a module over $C^{\infty}(M)$. Let us assume that we have a set of $N$ independent Killing sections; $u_{\alpha}^{a}(x)$. Then a general Killing section is a linear combination of the form
\begin{equation*}
u^{a}[\xi] = \xi^{\alpha}u_{\alpha}^{a},
\end{equation*}
\noindent where $(\xi^{a})$ are understood as vectors in $\mathbb{R}^{N}$ ($N \leq n(n+1)\slash 2$). As we have the structure of a Lie algebra here it must be the case that
\begin{equation*}
[u(1), u(2)] = u(3),
\end{equation*}
\noindent where we have defined $u(1) = \xi^{\alpha}_{1} u_{\alpha}^{a}s_{a}$ etc. It follows that we must have $\xi^{\alpha}_{3} = C_{\beta \gamma}^{\alpha} \xi^{\gamma}_{1} \xi^{\beta}_{2}$, i.e. a bi-linear combination of the vectors. Calculating the bracket  we see that
\begin{equation}
[u_{\alpha},  u_{\beta}]^{c} =  Q_{a}^{A}\left( u_{\alpha}^{a} \frac{\partial u_{\beta}^{c}}{\partial x^{a}} - \frac{\partial u_{\alpha}^{c}}{\partial x^{A}} u_{\beta}^{a}\right) - Q_{ba}^{c} u_{\alpha}^{a} u_{\beta}^{b}
= C_{\alpha \beta}^{\gamma}u_{\gamma}^{c},
\end{equation}
\noindent and so $C_{\alpha \beta}^{\gamma}$ is understood as the structure constant of the Lie algebra of killing sections, which we will denote as $\textnormal{iso}(E,G)$. Integrating this Lie algebra gives us the \emph{isometry group} of $(E,G)$, which we will denote $\textnormal{Iso}(E,G)$. The action of the group on $E$ is via integration of the vector fields defined by the lift of Killing sections (\ref{eqn;lift}). This lifted action will act as Lie algebroid automorphisms that furthermore preserve the Riemannian structure.

\smallskip

\noindent \textbf{An electromagnetic analogy:} Following the classical case, see Wald \cite{Wald:1984} [Appendix C], we have nice relation between electromagnetic theory and Killing sections. Consider a Line bundle $L \rightarrow M$ and a Lie algebroid $E \rightarrow M$. Linear connections on $L$ with values in $E$, often called E-connections, are formally similar to the  standard electromagnetic potential. In particular we have the \emph{gauge potential} $A_{a}$ defined (locally at least) by the covariant derivative
\begin{equation*}
\nabla^{L}_{a}\sigma = Q_{a}^{A}\frac{\partial \sigma}{\partial x^{A}} + \sigma A_{a},
\end{equation*}
\noindent where $\sigma$ is the component of a section with respect to some basis. Under a gauge transformation $\sigma \rightarrow \lambda \: \sigma$, where $\lambda \in C^{\infty}(M)$ is a nowhere vanishing function the covariant derivative transforms as $\nabla^{L}_{a}(\lambda \sigma) = \lambda \nabla_{a}^{L}\sigma$. This requires
\begin{equation*}
A_{a} \mapsto A^{\lambda}_{a} = A_{a} - Q_{a}^{A}\frac{\partial \lambda}{\partial x^{A}}.
\end{equation*}
The \emph{field strength} is given by
\begin{equation}
F_{ab} = [\nabla_{a}^{L},\nabla_{b}^{L} ] = \nabla_{a}A_{b} - \nabla_{b}A_{a}.
\end{equation}
The non-degeneracy of the Riemannian metric implies the following:
\begin{proposition}
There is a one-to-one correspondence between sections of a Riemannian Lie algebroid $(E,G)$ and E-connections on the trivial line bundle  $L = \mathbb{R}\times M$.
\end{proposition}
As standard we define the components of the \emph{Ricci tensor} as  $R_{ab} = R_{a\:\: cb}^{\:\: c}$, which of course can invariantly be defined. Via direct computation we arrive at:
\begin{theorem}
If $u \in \Sec(E)$ is a Killing section of a Riemannian Lie algebroid $(E,G)$, then
\begin{enumerate}
\item $\displaystyle G^{cb}(\nabla_{b}\nabla_{c}u)^{a} + u^{b}R_{bc}G^{cd}=0$,
\item $\displaystyle (\nabla_{a} u)^{a}=0$.
\end{enumerate}
\end{theorem}
\smallskip

\noindent \textbf{Statement} The above equations are almost formally identical to the source free Maxwell equations ``$\nabla_{\mu}F^{\mu\nu} =0$'' in the Lorentz gauge ``$\nabla_{\mu} A^{\mu}=0$''. The only real problem is the sign associated with the Ricci tensor term. On Ricci flat Riemannian Lie algebroids (i.e. $R_{ab}=0$), Killing sections satisfy the `source free Maxwell equations in the Lorentz gauge'  as defined above.

\newpage

\noindent \textbf{Killing--St\"{a}ckel tensors:}  The notion of Killing--St\"{a}ckel tensor follows directly from Theorem \ref{thm:killing1}.
\begin{definition}
Let us fix some Riemannian Lie algebroid $(E,G)$. Then a symmetric tensor $K \in C^{\infty}(E^{*})$ is said to be a \emph{Killing--St\"{a}ckel tensor} if and only if
\begin{equation*}
\{K, \mathcal{H}\}_{E} =0.
\end{equation*}
\end{definition}
The above definition reduces to a standard Killing--St\"{a}ckel field when $E = \sT M$ and hence the nomenclature.  Similar to the standard case, Killing--St\"{a}ckel tensors correspond to conserved quantities along the cogeodesic flow generated by $\mathcal{H}$. In the physical language they correspond to `generalised hidden symmetries'.

An immediate consequence of the definition  are  the following result:
\begin{proposition}
For a given Riemannian Lie algebroid, the set of Killing--St\"{a}ckel tensors form a  Poisson subalgebra  of $(C^{\infty}(E^{*}) , \{ \bullet, \bullet\}_{E})$.
\end{proposition}
There is no guarantee that a non-zero Killing sections exist in general. However,  the Poisson algebra of Killing--St\"{a}ckel tensor contains in addition to the zero section of $E$, (locally) constant functions  as zero degree Killing--St\"{a}ckel tensor and the energy $\mathcal{H}$ as a degree two Killing--St\"{a}ckel tensor.

\medskip
\noindent \textbf{Examples:} Here we briefly look at a few simple Riemannian Lie algebroids and the conditions for their sections to be Killing sections.
\begin{example}
If $(M,g)$ is a Riemannian manifold, then the tangent bundle is tautologically a Riemannian Lie algebroid. The cogeodesic equations are
\begin{align*}
& \dot{x}^{\mu} = g^{\mu \nu}p_{\nu},
& \dot{p}_{\nu} = - \frac{1}{2}\frac{\partial g^{\mu \rho}}{\partial x^{\nu}}p_{\rho}p_{\mu}.
\end{align*}
The Killing condition reduces to the standard Killing equation
\begin{equation*}
X^{\mu}\frac{\partial g_{\nu \rho}}{\partial x^{\mu}} + \frac{\partial X^{\mu}}{\partial x^{\nu}}g_{\mu \rho} + \frac{\partial X^{\mu}}{\partial x^{\rho}}g_{\mu \nu} = 0.
\end{equation*}
\end{example}
\begin{example}\label{ex:lie algebra}
Consider a Lie algebra $(\mathfrak{g},~ [\bullet,\bullet])$ understood as a Lie algebroid over a point. The  Killing condition on an element $ u \in \mathfrak{g}$   is just the algebraic condition
\begin{equation*}
u^{b}\left( Q_{ba}^{d}G_{dc} + Q_{bc}^{d}G_{da} \right)=0,
\end{equation*}
\noindent where  $Q_{ba}^{c}$ is the structure function. This condition can be written in an invariant way as
\begin{equation*}
\langle [u,v]| w\rangle  + \langle v| [u,w]\rangle =0,
\end{equation*}
\noindent for arbitrary $v$ and $w \in \mathfrak{g}$. This is condition is interpreted as the metric being $\ad_{u}$ invariant.  For a semisimple Lie algebra we have the natural choice of using the Killing form as the Riemannian metric, $G_{ab} = Q_{ac}^{d}Q_{db}^{c}$. In this case, the Riemannian metric is an invariant polynomial and all elements of the Lie algebra are Killing.
\end{example}
\begin{example}
Given a vector field $X \in \Vect(M)$ one can associate with it a Lie algebroid structure on the trivial line bundle $\mathbb{R} \times M \rightarrow M$ in the following way. First note that $\Sec(\mathbb{R} \times M) = C^{\infty}(M)$. The Lie bracket on functions on $M$ is given by $[f,g] = f X[g] - X[f]g$. The anchor is simply multiplication by $X$. We can equip this Lie algebroid with a Riemannian metric simply by specifying a strictly positive constant $c$ viz  $\langle f|g \rangle  = c \;fg$. Thus, without loss of generality we might as well set $c=1$. Meaning that the energy is specified by $\mathcal{H} = \half \pi^{2}$. The cogeodesic flow equations are just
\begin{align*}
& \dot{x}^{A} = X^{A}, & \dot{\pi} =0.
\end{align*}
The (co)geodesics are simply the integral curves of $X$. The Killing condition is also particularly simple
\begin{equation*}
X[f] =0.
\end{equation*}
\end{example}
\begin{example}
Consider an integrable distribution $E \subset \sT M$. Then as is well-know we have the structure of a Lie algebroid where the anchor is the natural inclusion and the Lie bracket is the restriction of the canonical Lie bracket on vector fields.  Furthermore, an integrable distribution is equivalent to a (smooth and regular) foliation of $M$ with leaves $\Sigma  \hookrightarrow M$ such that $\sT_{p}\Sigma \simeq E_{p} $ for all $p \in M$. A Riemannian  metric on $E$ defines a Riemannian metric on each of the leaves of this foliation.

Let us for simplicity consider a trivial foliation $M = \Sigma \times N$; this of course will serve as a local model. Then $E \simeq \sT \Sigma \times N$. We now employ local coordinates $(x^{a}, \dot{x}^{b}, y^{i})$ on $E$, where $(y^{i})$ serve as local coordinates on $N$. One can quickly convince oneself that the generalised geodesic equations are
\begin{align*}
&\frac{\rmd x^{a}}{\rmd t} = \dot{x}^{a},&  & \frac{\rmd y^{i}}{\rmd t }=0,&
& \frac{\rmd \dot{x}^{a}}{\rmd t } + \Gamma_{bc}^{a} \dot{x}^{c} \dot{x}^{b}=0,&
\end{align*}
\noindent thus we see that we have the  standard geodesic equation on $\Sigma$.

Sections of $E$ are vertical vector fields with respect to the fibration $\Sigma \times N \rightarrow N$. Thus, Killing sections satisfy the standard Killing equation on $\Sigma$ where the coordinates $y^{i}$ appear as `extra parameters'.
\end{example}

\section{Sigma models with Riemannian Lie algebroid targets}\label{sec:SigmaModel}
We now turn or attention to a generalistaion of standard sigma models where the target Riemannian manifold gets replaced with a Riemannian Lie algebroid. In particular, we will show how Killing sections are related to the internal symmetries of these models. We will draw heavily on the work of Mart\'{\i}nez \cite{Martinez:2005} throughout this section calling upon his results as needed. For an introduction to sigma models the reader can consult section 2 of Ketov's book \cite{Ketov:2000}.

\smallskip

\noindent \textbf{The sigma model:} Consider the following diagram\\
\begin{center}
\leavevmode
\begin{xy}
(0,20)*+{\sT \Sigma}="a"; (20,20)*+{E}="b";%
(0,0)*+{\Sigma}="c"; (20,0)*+{M}="d";%
(40,20)*+{\sT M}="e";%
{\ar "a";"b"}?*!/_3mm/{\Phi};
{\ar "a";"c"}?*!/^3mm/{\tau_{\Sigma}};{\ar "b";"d"}?*!/^3mm/{\pi};%
{\ar "c";"d"} ?*!/^3mm/{\phi};%
{\ar "b"; "e"}?*!/_3mm/{\rho};%
{\ar "e"; "d"}?*!/_3mm/{\tau};%
\end{xy}
\end{center}
\medskip

We will insist that the map $\Phi: \sT \Sigma \rightarrow E$ is a morphisms of Lie algebroids over $\phi : \Sigma \rightarrow M$. Let us pick local coordinates $(x^{A}, y^{a})$ on $E$ and $(z^{i}, \delta z^{j})$ on $\sT \Sigma$. Then we write the components of $\Phi$ as
\begin{equation}
\Phi^{*}(x^{A}, y^{a}) = (\phi^{A}(z) , \delta z^{i}\chi_{i}^{a}(z)).
\end{equation}
Thus, we can employ `local coordinates' $(\phi^{A}(z), \chi_{i}^{a}(z))$ on the infinite dimensional manifold of all vector bundle morphism  from $\sT \Sigma$ to $E$.  The condition that we have a morphism of Lie algebroids is locally given by
\begin{align}\label{eqn:morphism condition}
 & \frac{\partial \phi^{A}}{\partial z^{i}}(z) =  \chi^{a}_{i}(z)Q_{a}^{A}(\phi(z)),
        & \chi_{j}^{b}(z)\chi_{k}^{c}(z)Q_{cb}^{a}(\phi(z)) =  \frac{\partial \chi_{k}^{a}}{\partial z^{j}}(z) -\frac{\partial \chi_{j}^{a}}{\partial z^{k}}(z).
\end{align}

Note that a Lie algebroid morphism  is then automatically  \emph{admissible}  in the sense that
\begin{equation*}
\sT \phi = \rho \circ \Phi.
\end{equation*}

It seems essential for a variational principle that we consider Lie algebroid morphisms $ \Phi: \sT \Sigma \rightarrow E$, rather than morphisms between general Lie algebroids. If $\Sigma = \mathbb{R}$ with the standard Euclidean metric, then we are discussing standard first order mechanics on the Lie algebroid $E$.

To construct the model we now assume that our `space-time' is a Riemannian manifold $(\Sigma, g)$.  The target Lie algebroid is then assumed to be a Riemannian Lie algebroid $(E, G)$. Then we have;
\begin{definition}
Under the above stipulations, the \emph{Lie algebroid sigma model} is defined by the action
\begin{eqnarray}
S[\Phi] &:=& \int_{\Sigma} \rmd Vol_{g}\: \textnormal{Tr}(g^{-1}\cdot \Phi^{*}G)\\
\nonumber &=& \frac{1}{2} \int_{\Sigma}\rmd \textbf{z}\:  \sqrt{|g|}(z) \:\:  g^{ij}(z)\chi_{j}^{b}(z) \chi_{i}^{a}(z)G_{ab}(\phi(z)).
\end{eqnarray}
\end{definition}
\begin{remark}
One could also add a topological term which we omit from our discussion for the moment. We will not discuss how to add  a potential and so we will not consider Landau--Ginzburg-like models. Furthermore we will neglect boundary terms, or simply assume that $\Sigma$ has no boundary.
\end{remark}

\smallskip

\noindent \textbf{The field equations:}  Following Mart\'{\i}nez \cite{Martinez:2005} the Euler--Lagrange equations for models like the Lie algebroid sigma model are:
\begin{equation*}
Q_{a}^{A}(\phi(z))\frac{\partial \mathcal{L}}{\partial \phi^{A}} - \frac{\partial}{\partial z^{i}}\left( \frac{\partial \mathcal{L}}{\partial \chi_{i}^{a}}\right) + \chi_{i}^{b}Q_{ba}^{c}(\phi(z)) \left(\frac{\partial \mathcal{L}}{\partial \chi_{i}^{c}}\right) =0.
\end{equation*}
In full, the morphism condition and the equations of motion for the Lie algebroid sigma model are:
\begin{subequations}
\begin{align}
& \frac{\partial \phi^{A}}{\partial z^{i}}(z) = \chi_{i}^{a}(z)Q_{a}^{A}(\phi(z)),\label{eqn:EQMa}\\
&  \chi_{j}^{b}(z)\chi_{i}^{c}(z)Q_{cb}^{a}(\phi(z)) = \frac{\partial \chi_{i}^{a}}{\partial z^{j}}(z) ~{-}~\frac{\partial \chi_{j}^{a}}{\partial z^{i}}(z),\label{eqn:EQMb}\\
& \frac{1}{\sqrt{|g|}(z)} \frac{\partial }{\partial z^{j}}\left( \sqrt{|g|}(z)g^{ji}(z)\chi_{i}^{a}(z)\right) + g^{ij}(z)\chi_{j}^{c}(z)\chi_{i}^{b}(z)\Gamma_{bc}^{a}(\phi(z)) =0.\label{eqn:EQMc}
\end{align}
\end{subequations}
Clearly the critical points of the Lie algebroid sigma model correspond to a generalised notion of harmonic maps,  or `instantons' in the physics language. The Lie algebroid generalisation of the \emph{tension field} of a configuration $\tau(\Phi) \in \Sec(\phi^{*}E)$ is given by (\ref{eqn:EQMc}). The tension field can be interpreted as giving the `generalised direction' that $\Sigma$ wants to move in $M$ in order to minimise the action.  Once at a critical point the tension field is the zero section and the configuration does not want to move any further.
\begin{definition}
Let $(E,G)$ be a Riemannian Lie algebroid an let $(\Sigma, g)$ be a Riemannian manifold. The a Lie algebroid morphism $\Phi: \sT \Sigma \rightarrow E$, (over $\phi: \Sigma \rightarrow M $) is said to be a \emph{(generalised) harmonic map} if it is  a critical point of the Lie algebroid sigma model.
\end{definition}
If $E = \sT M$, then we get the standard notion of a harmonic map between $(\Sigma, g)$ and $(M, G)$. If $M = \mathbb{R}$ equipped with the standard metric, then we have the notion of a harmonic function on $\Sigma$. If $\Sigma = \mathbb{R}$, again equipped with the standard metric,  then the equations of motion are just the generalsied geodesic equations (\ref{eqn:geodesic equations}).
\smallskip

\noindent \textbf{Symmetries and Killing sections:} Consider the following  infinitesimal field redefinitions
\begin{align}\label{eqn:redef}
& \phi^{A} \mapsto \phi^{A} + \zx^{\alpha} u^{a}_{\alpha}(\phi)Q_{a}^{A}(\phi), &  \chi_{i}^{a} \mapsto  \chi_{i}^{a} + \zx^{\alpha} \chi_{i}^{b}\left( Q_{b}^{A}(\phi)\frac{\partial u^{a}_{\alpha}}{\partial \phi^{A}}(\phi) - u^{c}_{\alpha}(\phi)Q_{cb}^{a}(\phi) \right),&
\end{align}
\noindent here $\zx^{\alpha} \in \mathbb{R}^{n}$ for some $n$ and $u^{a}_{\alpha}$ are sections of  $\phi^{*}E$.  A direct calculation shows that
\begin{equation*}
S \mapsto S + \int_{\Sigma}\rmd Vol_{g}\: \textnormal{Tr}\left(g^{-1}\cdot \Phi^{*}\mathcal{L}_{u[\zx]}G\right).
\end{equation*}
Clearly, the Lagrangian is invariant under these field redefinitions if $u_{\alpha}^{a}$ are Killing sections of $(E,G)$. Comparing the field redefinitions with the lift of a section to a vector field (\ref{eqn;lift}) and the fact that this lift is a morphism of Lie algebras, we conclude that the Lie algebra formed by these field redefinitions is closed on-shell and off-shell; i.e. closed irrespective of the equations of motion.
\begin{proposition}
The Lie algebra of the internal symmetries of the Lie algebroid sigma model is given by the Lie algebra of Killing sections  $\textnormal{iso}(E,G)$.
\end{proposition}
It is not hard to show that the associated Noether currents are given by (up to standard ambiguities)
\begin{equation*}
J^{j}_{\zx} = \zx^{\alpha}u_{\alpha}^{a}g^{ji}\chi_{i}^{b}G_{ba}.
\end{equation*}
\begin{example}
If $E= \sT M$, then the sigma model reduces to the standard sigma models (without topological term and potential). We then see that the field redefinition is just $\phi^{A} \mapsto \phi^{A} + \zx^{\alpha}X_{\alpha}^{A}$, where $X_{\alpha}^{A}$ are standard Killing vector fields on $M$.
\end{example}

\smallskip

\noindent \textbf{Adding a Wess--Zumino term:} Similar to the standard case, our `branes' can be charged and couple to a  Lie algebroid k-form, assuming that $\dim \Sigma =k$. The k-form is understood as a  electromagnetic-like potential or simply a bulk interaction term. Let us pick local coordinates $(x^{A}, \theta^{a})$ on $\Pi E$; we are now in the world of supermanifolds and the $\theta$ coordinates are anticommuting. Here $\Pi$ is the parity reversion functor that acts by formally shifting the Grassmann parity of the fibre coordinates of the vector bundle $E$. Then, any Lie algebroid k-form looks like
\begin{equation*}
C = \frac{1}{k!} \theta^{a_{1}}\cdots \theta^{a_{k}}C_{a_{k} \cdots a_{1}}(x).
\end{equation*}
This Lie algebroid form can be pull-backed to a  k-form on $\Sigma$ using $\Pi \Phi : \Pi \sT \Sigma \rightarrow \Pi E$. In local coordinates we have
\begin{equation*}
(\Pi \Phi)^{*}C =  \frac{1}{k!} \rmd z^{i_{1}}\cdots \rmd z^{i_{k}}\chi_{i_{1}}^{a_{1}}\cdots \chi_{i_{k}}^{a_{k}}C_{a_{k} \cdots a_{1}}(\phi(z)).
\end{equation*}
\noindent Here we consider $\rmd z$ to be the fibre coordinates of $\Pi \sT \Sigma$ and so we do not include a wedge product in our notation. The Lie algebroid sigma model in the presence of a background k-form is
\begin{equation}
S[\Phi] = \int_{\Sigma} \rmd Vol_{g}\: \textnormal{Tr}(g^{-1}\cdot \Phi^{*}G) + \int_{\Sigma} (\Pi \Phi)^{*}C.
\end{equation}
\noindent In the above the integration of the topological term is technically the Berezin integration on the antitangent bundle of $\Sigma$.

\begin{example}
Consider the case where $\Sigma = \mathbb{R}$ and the action is given by
\begin{equation*}
S[\Phi] = \int_{\mathbb{R}} \rmd t\: \left( \frac{1}{2}\chi^{a}(t)\chi^{b}(t)G_{ba}(\phi(t))  +  \chi^{a}(t)C_{a}(\phi(t))\right).
\end{equation*}
\noindent it is not hard to see that the equations of motion are:
\begin{align*}
& \frac{\rmd \phi^{A}}{\rmd t}(t) = \chi^{a}Q_{a}^{A}(\phi(t)),   & & \frac{\rmd \chi^{a}}{\rmd t}(t) + \chi^{b}(t)\chi^{c}(t)\Gamma_{cb}^{a}(\phi(t)) + \chi^{b}(t)\mathbb{F}_{bc}(\phi(t))G^{ca}(\phi(t)) =0,
\end{align*}
\noindent where the curvature  is defined as $\mathbb{F} = \rmd_{E}C$, where $\rmd_{E} \in \Vect(\Pi E)$ is the Lie algebroid differential considered as a homological vector field. This system  represents a Lie algebroid version of a charged particle moving on a curved Riemannian manifold.  As standard there is arbitrariness in selecting the one-form  as  $C \mapsto C + \rmd_{E}f$ for any $f\in C^{\infty}(M)$ does not effect the equations of motion.  If $C$ is $\rmd_{E}$-closed then $(E, C)$ is a Jacobi algebroid (c.f. \cite{Grabowski:2001,Iglesias-Ponte:2000}) and the curvature term in the equations of motion vanishes.
\end{example}
\smallskip
\noindent \textbf{Symmetries of a charged particle on a Riemannian Lie algebroid:} Let us concentrate on the model given in the previous example. It is not hard to show that under the field redefinitions (\ref{eqn:redef}) that the action transforms as
\begin{equation*}
S \mapsto S + \int_{\mathbb{R}}\left(\Phi^{*}(\mathcal{L}_{u[\zx]}G)  + (\Pi \Phi)^{*}(\mathcal{L}_{u[\zx]}C)  \right),
\end{equation*}
\noindent where the Lie derivative of $C$ is defined using the lift of a section to a vector field (\ref{eqn;lift}). As we have a one-form there are no sign complication in this definition. We then see that this action is still invariant under the field redefinitions if $u[\zx]$ are Killing sections and in addition $\mathcal{L}_{u[\zx]}C =0$. The associated Noether current, similarly to the classical case is given by
\begin{equation*}
J_{\zx} = \zx^{\alpha}u_{\alpha}^{a}\left(\chi^{b}G_{ba}  + C_{a}\right).
\end{equation*}

\section{Concluding remarks}\label{sec:conclusion}
In this paper we showed how the notion of a Killing vector field on a Riemannian manifold generalises to Lie algebroids equipped with Riemannian structures. We then showed that several of the common equations that express the Killing condition directly generalise to the setting of Lie algebroids. The theory of Riemannian structures largely generalises to Lie algebroids with little fuss. We then applied the technology developed to study the internal symmetries of (classical) sigma models that have Riemannian Lie algebroids as targets. We have focused on mathematical questions and have not attempted to  find phenomenological applications of these models. Moreover, we have not examined the important, but separate question of explicitly finding examples of (Lie algebroid) harmonic maps.

\section*{Acknowledgments}\label{sec:Ack}
The author thanks  M.~Anastasiei for his comments on earlier drafts of this paper.


\noindent Andrew James Bruce\\
\emph{Institute of Mathematics, Polish Academy of Sciences,}\\ {\small \'Sniadeckich 8,  00-656 Warszawa, Poland}\\ {\tt andrewjamesbruce@googlemail.com}

\end{document}